\documentclass[11pt]{article}
\usepackage{amsmath, amssymb, amsthm, mathrsfs, commath, mathtools}
\usepackage[headheight=0pt,headsep=0pt,textwidth=6in,textheight=9in]{geometry}
\usepackage[dvipsnames]{xcolor}

\usepackage{indentfirst}

\usepackage[margin=1cm,font=small]{caption}

\usepackage{tikz}
\usetikzlibrary{decorations.markings}
\usepackage{float}

\usepackage{chngcntr}

\newcommand{\R}{\mathbb{R}}

\newcommand{\C}{\mathbb{C}}
\newcommand{\Z}{\mathbb{Z}}
\newcommand{\N}{\mathbb{N}}

\renewcommand{\phi}{\varphi}

\renewcommand{\d}{\, d}									               
\newcommand{\tn}[1]{\textnormal{#1}}					                   
\renewcommand{\l}{\left}                                               
\renewcommand{\r}{\right}								               
\renewcommand{\b}[1]{\l( #1 \r)}						               
\newcommand{\st}{\,\colon\,}							              	   
\newcommand{\eq}[1]{\begin{equation} #1 \end{equation}}
\newcommand{\aeq}[1]{\begin{align}\begin{split} #1 \end{split}\end{align}}

\counterwithin{equation}{section}

\theoremstyle{plain}
\newtheorem{thm}[equation]{Theorem}
\newtheorem{lem}[equation]{Lemma}

\newtheorem{prop}[equation]{Proposition}
\newtheorem{cor}[equation]{Corollary}

\theoremstyle{definition}
\newtheorem{defn}[equation]{Definition}
\newtheorem{exa}[equation]{Example}

\newtheorem{opr}[equation]{Open Problem}

\title{Surfaces with Commuting Boundary Laplacian and Dirichlet-to-Neumann Map}
\author{Romain Speciel}
\date{\today}

\begin{document}
\maketitle


\abstract{
For $M\subset \R^{d\geq 3}$ a smooth, connected, compact $d$-dimensional submanifold with boundary, equipped with the standard metric, the Laplacian on $\partial M$ is known to commute with the corresponding Dirichlet-to-Neumann map if and only if $M$ is a ball. In this paper, we investigate the $d=2$ case and show that, surprisingly, there exists a one-parameter family of submanifolds of $\R^2$ as above for which the boundary Laplacian and the Dirichlet-to-Neumann map commute, thus answering an open problem posed by Girouard, Karpukhin, Levitin, and Polterovich. We then classify all such Riemannian surfaces of genus $0$ or whose boundary has $k\geq 3$ connected components.
}

\section{Introduction}
\label{sec: intro}

Let $(M,g)$ be a smooth, connected, oriented, compact Riemannian manifold of dimension $d\geq 2$ with nonempty boundary. To study the positive Laplacian $\Delta_M$ on $M$, we must impose boundary conditions in order to obtain a well-posed problem. This is most commonly done via either \textit{Dirichlet} or \textit{Neumann} boundary conditions, where one prescribes either the boundary value or the boundary normal derivative of functions on the interior, respectively. The \textit{Dirichlet-to-Neumann map}, which we denote by $\Lambda\colon C^\infty(\partial M)\to C^\infty(\partial M)$, is defined to correspond between these two classical boundary conditions by assigning to a function $f\in C^\infty(\partial M)$ the function $\Lambda f\coloneqq \partial_n u |_{\partial M}$, where $n$ is the outward-pointing unit normal at the boundary and $u\in C^\infty(M)$ solves
\eq{
	\Delta_M u =0\quad\tn{and}\quad u|_{\partial M}=f.
}
In other words, the elliptic boundary value problems
\eq{
	\begin{cases}
		\Delta_M u =0\\
		u|_{\partial M}=f
	\end{cases}
	\quad \tn{and}\quad 
	\begin{cases}
		\Delta_M v =0\\
		\partial_n v |_{\partial M}=\Lambda f
	\end{cases}
}
have the same solution. The Dirichlet-to-Neumann map is a self-adjoint pseudodifferential operator of order 1  which is central to many problems in inverse geometry, and has applications across a broad range of fields including physics, medical imaging and geology (we refer the reader to \cite{U14} for an overview). 

\begin{exa}
	Let $M\subset \R^{d\geq 2}$ be the unit ball equipped with the standard metric, with boundary $\partial M =\mathbb{S}^{d-1}$. Define $P_k$ to be the space of homogeneous harmonic polynomials in $\R^d$ of degree $k$, and recall that, in spherical coordinates $(r,\omega)$, the Laplacian is given by
	\eq{
	\label{eq: laplacian in polar coord}
		\Delta_{\R^d}=-\b{\partial_r^2+\frac{d-1}{r}\partial_r-\frac{1}{r^2}\Delta_{\mathbb{S}^{d-1}}}
	}
	so for $p(r,\omega)\in P_k$,
	\eq{
		0=\b{k(k-1)+k(d-1)-\Delta_{\mathbb{S}^{d-1}}}p(1,\omega).
	}
	Therefore, the space $H_k=\{p(1,\omega)\st p\in P_k\}$ consists of eigenfunctions of $\Delta_{\mathbb{S}^{d-1}}$ corresponding to the eigenvalues $k(k+d-2)$. Furthermore, $\Lambda$ acts on $H_k$ simply as multiplication by $k$ (since it corresponds to $\partial_r|_{r=1} p(r,w)$), and one can show using the Stone-Weierstrass theorem that $L^2(\mathbb{S}^{d-1})=\bigoplus_kH_k$. Conclude
	\eq{
	\label{eq: spherical laplacian in terms of DtoN}
		\Delta_{\mathbb{S}^{d-1}}=\Lambda^2+(d-2)\Lambda.
	}
	Consult \cite{CGGS23} for further examples as well as an excellent survey of the subject.
\end{exa}

In general, the Dirichlet-to-Neumann map depends intimately on the geometry of $M$, yet the relationship displayed in Equation (\ref{eq: spherical laplacian in terms of DtoN}) always holds \textit{to leading order}. This is rigorously expressed by the fact that, for any $M$, we have
\eq{
	\sigma_2(\Delta_{\partial M})=\sigma_2(\Lambda^2)=\sigma_2(\Lambda^2+(d-2)\Lambda),
}
where $\sigma_k$ denotes the symbol of order $k$ (see Chapter 12C in \cite{T96} for details). It is natural to then ask the following more refined question: when is the boundary Laplacian given exactly by a function of the Dirichlet-to-Neumann map?

To answer this, we instead consider a more general problem by observing that if the boundary Laplacian is a function of the Dirichlet-to-Neumann map, then we must have $[\Delta_{\partial M},\Lambda]=0$ (though the converse implication is not necessarily true), and thus ask when this commutativity property holds. This problem is stated and investigated in \cite{GKLP22}, where the case of Euclidian submanifolds in dimensions $d\geq 3$ is addressed, but the two dimensional case is not covered by the presented methods and is thus posed as an open problem by the authors.

In this paper, we address the $d=2$ case in broad generality. Our main result consists of a  classification of all surfaces $M$ as above, of genus $0$ or with $k\geq 3$ boundary connected components, for which the boundary Laplacian and the Dirichlet-to-Neumann map commute. To state the theorem, we first define the following equivalence, discussed in \cite{CGGS23}:

\begin{defn}
	We say $(M_1,g_1)$ and $(M_2,g_2)$ are \textit{conformal} to each other if there exists a function $\phi\in C^\infty (M_1)$ such that $(M_1,e^{2\phi}g_1)$ and $(M_2,g_2)$ are isometric to each other, and \textit{$\sigma$-isometric} if $\phi$ can be taken to vanish on $\partial M_1$.
\end{defn} 

This equivalence is natural to consider due to the conformal covariance of the Laplacian in dimension two. We suppress the metric when it is understood. Now, a quick argument shows that if the boundary Laplacian and the Dirichlet-to-Neumann map commute for a surface $M_1$, and $M_2$ is $\sigma$-isometric to $M_1$, then $M_2$ enjoys this same commutator property (see Section \ref{sec: simply connected} for details). Our main theorem is then:

\begin{thm}
\label{thm: main theorem}
	Let $M$ be a smooth, connected, oriented, compact Riemannian surface with nonempty boundary. Suppose further that $M$ is either of genus $0$ or that $\partial M$ has $k\geq 3$ connected components. Then, $[\Delta_{\partial M},\Lambda]=0$ if and only if $M$ is $\sigma$-isometric to a disc, a logarithmic oval or a flat cylinder.
\end{thm}

It follows from our result that, surprisingly, there exists a one parameter family of planar submanifolds, which we call \textit{logarithmic ovals}, for which the boundary Laplacian and the Dirichlet-to-Neumann map commute. These correspond (up to translation, rotation and scaling) to images of the closed unit disc $\overline{\mathbb{D}}\subset \C\cong \R^2$ under maps of the form $\psi(z)=\log(z-a)$, where $a>1$ and the logarithm is taken with branch cut along the positive real axis. Note that as $a\to \infty$, this family limits to the disc, after appropriate rescaling. In fact, the theorem indicates that these are the \textit{only} planar submanifolds with this commutativity property. We state this as a corollary.

\begin{cor}
	Let $M\subset \R^2$ be a smooth, connected, compact submanifold with boundary, equipped with the standard metric, such that $[\Delta_{\partial M},\Lambda]=0$. Then, $M$ is a disc or a logarithmic oval.
\end{cor}

To prove our theorem, we first address the simply connected case in Section \ref{sec: simply connected}, where our main insight is to notice that the commutativity assumption in fact implies that the reciprocal of the conformal factor restricted to the boundary has finite Fourier support, from which the result follows. Curiously, this approach is slightly similar to Hurwitz's proof of the isoperimetric inequality as presented in \cite{SS05}, in that a geometric condition is deduced from a bound on the Fourier support of a related function. Next, in Section \ref{sec: higher connectivity}, we extend these methods to the doubly connected case, and treat the case of higher connectivity by an analysis of the nodal set of harmonic extensions of eigenfunctions of the boundary Laplacian.

Finally, we also obtain a generalization of a result from \cite{GKLP22} by Girouard, Karpukhin, Levitin and Polterovich:
\begin{thm}\tn{(cf. with Theorem 1.3 in \cite{GKLP22})}
\label{thm: GKLP improvement}
	Let $M\subset \R^{d}$ be a smooth, connected, oriented, compact manifold of dimension $d\geq 3$, equipped with the standard metric. Then, $[\Delta_{\partial M}, \Lambda]=0$ if and only if $M$ is a ball.
\end{thm}

This result is proved in \cite{GKLP22}, under the additional assumption that the boundary is connected, by computing the symbol of the commutator in geometric terms to deduce that the mean curvature of $\partial M$ must be constant, and then applying Alexandrov's theorem about embedded constant mean curvature surfaces. (Note, however, that this proof does not cover the $d=2$ case since in the computations appears a factor of $(d-2)$, which vanishes in dimension two.) We show that the boundary connectedness assumption may in fact be removed as a consequence of Proposition \ref{prop: boundary isospectrality}.

Ou approach does not cover the case when $M$ has two or fewer boundary components and nonzero genus. We state this as an open problem:

\begin{opr}
	Let $M=(\mathbb{T}_2\#\dots\# \mathbb{T}_2)\setminus D$, where $D$ is either a disc or the disjoint union of two discs. Does there exist a metric on $M$ such that $[\Delta_{\partial M}, \Lambda]=0$?
\end{opr}

\vspace{3em}

\noindent \textbf{Acknowledgements:} The author thanks Iosif Polterovich and David Sher for suggesting this problem and engaging in stimulating conversations, Josef Greilhuber for his significant contributions to the ideas of Section \ref{sec: higher connectivity}, and Rafe Mazzeo for his expert and generous advising throughout. This work was in part supported by an NSERC-PGSD grant.

\section{Simply-Connected Surfaces}
\label{sec: simply connected}

We begin by recalling how the concerned geometric objects change under conformal variations of the metric. To this end, take $(M,g)$ as before and $\phi\in C^{\infty}(M)$, and compare the metrics $g_0=e^{2\cdot 0}g$ and $g_\phi=e^{2\phi}g$ on $M$. With subscripts corresponding to the conformal factor, we have
	\eq{
		\Delta_{(M,g_\phi)}=e^{-2\phi}\Delta_{(M,g_0)}, \quad \nabla_\phi  =e^{-2\phi}\nabla_0\quad \tn{and}\quad n_\phi=e^{-\phi}n_0.
	}
	It follows that harmonic functions are preserved under conformal changes of the metric and that
	\eq{
		\partial_{n_{\phi}}u=g_\phi(\nabla_\phi u,n_\phi)=e^{-\phi}g_0(\nabla_0u,n_0)=e^{-\phi}\partial_{n_0}u,
	}
	so $\Lambda_\phi=e^{-\phi}\Lambda_0$. Note this property is unique to dimension two; in higher dimensions, the Dirichlet-to-Neumann map has a more complicated variation under conformal perturbations of the interior since the Laplacian is no longer conformally covariant.
	
	Next, let $\partial_{t_\phi}$ denote the derivative along an oriented unit-speed parametrization of the boundary with respect to the metric $g_\phi$. A quick computation shows
	\eq{
		\Delta_{(\partial M,g_\phi)}=-(\partial_{t_\phi})^2=-(e^{-\phi}\partial_{t_0})^2= -e^{-2\phi}(\partial_{t_0}^2-\phi'\partial_{t_0}),
	}
	from which we conclude
	\eq{
	\label{eq: conf commutator}
		[\Delta_{(\partial M,g_\phi)},\Lambda_\phi]=[-e^{-2\phi}(\partial_{t_0}^2-\phi'\partial_{t_0}),e^{-\phi}\Lambda_0].
	}
	In particular, this commutator is unchanged across pairs of surfaces which are $\sigma$-isometric, as mentioned in Section \ref{sec: intro}.

	We shall now assume $(M,g)$ to be simply connected. By the smooth Riemann mapping theorem, $M$ is conformal to the closed unit disc with the standard metric. Take therefore $\phi\in C^{\infty}(\overline{\mathbb{D}})$ so that $(\overline{\mathbb{D}},g_\phi=e^{2\phi}g_0)$ and $(M,g)$ are isometric, where $g_0$ is the standard Euclidian metric. 
	Our main insight is to realize that Equation \ref{eq: conf commutator} implies finite Fourier support of the reciprocal of the conformal factor restricted to the boundary. To this end, define for $f$ a function on $\mathbb{S}^1$ (or, by abuse of notation, on $\overline{\mathbb{D}}$ then restricted to $\mathbb{S}^1$) its Fourier coefficients $c_f(k)$ by
	\eq{
		c_f(k)=\frac{1}{2\pi}\int e^{-ik\theta}f(e^{i\theta})\d \theta.
	}

	\begin{prop}
	\label{prop: commutator means vanishing coefs}
		The commutator $[\Delta_{(\mathbb{S}^1,g_\phi)},\Lambda_\phi]$ vanishes if and only if $c_{e^{-2\phi}}(k)=0$ for $\abs{k}\geq 3$.
	\end{prop}

\begin{proof}	
	Let $\theta$ denote the standard unit speed parametrization of $\mathbb{S}^1$, and observe from Equation (\ref{eq: laplacian in polar coord}) that $u(r,\theta)=r^{\,\abs{n}}e^{in\theta}$ is the harmonic extension of $e^{in\theta}$ to the flat disc, so
	\eq{
		\Lambda_0e^{in\theta}=\partial_r|_{r=1}\b{ r^{\,\abs{n}}e^{in\theta}}=\abs{n}e^{in\theta}.
	}
	Now, from Equation (\ref{eq: conf commutator}), the commutator vanishes if and only if for every $n\in \Z$ we have
	\eq{
	e^\phi[e^{-2\phi}(\partial_\theta^2-\phi'\partial_\theta),e^{-\phi}\Lambda_0]e^{in\theta}=0,
	}
	which in turn holds if and only if the left hand side integrates against $e^{i(k-n)\theta}$ to zero for every $k\in \Z$. Combining these observations and expanding derivatives using the fact that the Dirichlet-to-Neumann map is self adjoint, we see that $[\Delta_{(\mathbb{S}^1,g_\phi)},\Lambda_\phi]=0$ if and only if
	\eq{
	\label{eq: expanded commutator}
		0=\int e^{ik\theta} e^{-2\phi}\b{\abs{n}(2(\phi')^2-\phi''-3in\phi'-n^2)-\abs{k-n}(-n^2-in\phi')}\d \theta
	}
	for every $n, k\in \Z$. Integrate by parts to obtain
	\eq{
		\int e^{ik\theta}e^{-2\phi}\phi'\d \theta=-\frac{1}{2}\int e^{ik\theta}\b{\partial_\theta e^{-2\phi}}\d \theta=\frac{ik}{2}\int e^{ik\theta} e^{-2\phi}\d \theta
	}
	and
	\eq{
		\int e^{ik\theta}e^{-2\phi}(2(\phi')^2-\phi'')\d \theta=\frac{1}{2}\int e^{ik\theta}\b{\partial_\theta^2 e^{-2\phi}}\d \theta=\frac{-k^2}{2}\int e^{ik\theta} e^{-2\phi}\d \theta,
	}
	and deduce (\ref{eq: expanded commutator}) is equivalent to
	\eq{
		0=\b{\abs{n}(-k^2/2+3kn/2-n^2)-\abs{k-n}(-n^2+nk/2)}\int e^{ik\theta}e^{-2\phi}\d \theta.
	}
	When $n\geq 0$ and $k\leq n$, or when $n\leq 0$ and $k\geq n$, this statement is vacuous since the leading coefficient is immediately zero. On the other hand, when $0\leq n\leq k$ or $k\leq n\leq 0$, we obtain
	\eq{
		0=n(k-n)(k-2n)\int e^{ik\theta}e^{-2\phi}\d \theta.
	}
	Conclude (\ref{eq: expanded commutator}) holds for every $n,k$ precisely when
	\eq{
		\int e^{ik\theta}e^{-2\phi}\d \theta=0
	}
	for $\abs{k}\geq 3$, as desired.
\end{proof}

Next, we seek to show that if $c_{e^{-2\phi}}(k)=0$ for $\abs{k}\geq 3$, then $(\overline{\mathbb{D}},g_\phi)$ is $\sigma$-isometric to either a disc or a logarithmic oval. We shall employ the following classical lemma, which explains how this condition of finite Fourier support may be interpreted. The proof is provided for completeness.	
	
\begin{lem}[Fej\'er-Riesz Theorem, 1916]
\label{lem: FR Them}
	Let $f\colon \mathbb{S}^1\to \C$ be smooth, nonvanishing, and suppose $c_{\,\abs{f}}(k)=0$ for $\abs{k}\geq n+1$. Then, there exists a polynomial $p$, nonvanishing on $\overline{\mathbb{D}}$ and of degree at most $n$, such that $\abs{f(e^{i\theta})}=\abs{p(e^{i\theta})^2}$.
\end{lem}
\begin{proof}
	Begin by writing $\abs{f(e^{i\theta})}=\sum_{k=-n}^n a_ke^{ik\theta}$. Since $\abs{f}$ is real valued, we must have $\overline {a_k}=a_{-k}$, so the function $w(z)=\sum_{k=-n}^n a_k z^k$ satisfies $w(z)=\overline {w(1/\overline{z})}$. We suppose without loss of generality that $a_{\pm n}\neq 0$ and set $q(z)=z^nw(z)$, which is then a polynomial of degree exactly $2n$ with $q(0)\neq 0$.  Furthermore, the roots of $q$ do not lie on the unit circle since $\abs{f}>0$, and hence come in distinct pairs $\alpha_k$ and $1/\overline{\alpha_k}$ with $\abs{\alpha_k}>1$. We can therefore rewrite
	\aeq{
		w(z)&=c\prod_{k=1}^n(z-\alpha_k)(z^{-1}-\overline{\alpha_k})
	}
	for some constant $c$. By setting $z=1$ above, we observe the product is positive and $w(1)=|f(1)|>0$, hence $c>0$. Define now $p(z)=\sqrt{c}\cdot \prod_{k=1}^n(z-\alpha_k)$ and conclude
	\eq{
		\abs{p(e^{i\theta})^2}=p(e^{i\theta})\cdot \overline{p(e^{i\theta})}=w(e^{i\theta})=\abs{f(e^{i\theta})},
	}
	as desired.
\end{proof}	
	
We are now aptly armed to attack our theorem.

\begin{proof}[Proof of Theorem \ref{thm: main theorem} in the simply connected case]
	As discussed, it suffices by the Riemann mapping theorem to show the result for $(\overline{\mathbb{D}},g_\phi=e^{2\phi}g_0)$, where $g_0$ is the standard metric. Suppose then that  $[\Delta_{(\mathbb{S}^1,g_\phi)},\Lambda_\phi]=0$, and apply Propositions 2.6 and Lemma 2.15 to produce a polynomial $p$ of degree at most $2$ which is nonvanishing on the unit disc and has $e^{-2\phi}=\abs{p^2}$ on $\mathbb{S}^1$. Define $\psi\colon \overline{\mathbb{D}}\to \C$ by
	\eq{
		\psi(z)=\int_0^z\frac{1}{p(w)}\d w,
	}
	and denote the image of $\psi$ by $\Omega$. We now consider three cases:
	
	\noindent Case 1: $p(z)=q(z)^2$ with $q$ of degree at most 1. After integrating, we deduce $\psi$ is a M\"obius transformation, hence $\Omega$ is necessarily a disc.
	
	\noindent Case 2: $p(z)=z-a$. After integrating, we deduce that, up to a constant scale and shift, $\psi(z)=\log (z-a)$ (and note that $\abs{a}>1$ since the roots of $p$ lie outside $\overline{\mathbb{D}}$, so the logarithm is well defined). Therefore, the image $\Omega$ corresponds uniquely (up to translations, rotations, and scaling) to the domains obtained by taking $a\in \R_{>1}$ and taking the branch of $\log$ along the positive real axis. The image of the disc under $\psi$ is a \textit{logarithmic oval}, which is not a disc, yet which has commuting Dirichlet-to-Neumann map and boundary Laplacian. Examples of such domains are displayed in Figure 1. Note that, as $a\to \infty$, the resulting domain tends to a disc (after appropriate translation and rescaling).

\begin{figure}[h!]
\label{fig: log ovals}
    \centering
    \begin{minipage}{0.45\textwidth}
        \centering
        \includegraphics[width=0.9\textwidth]{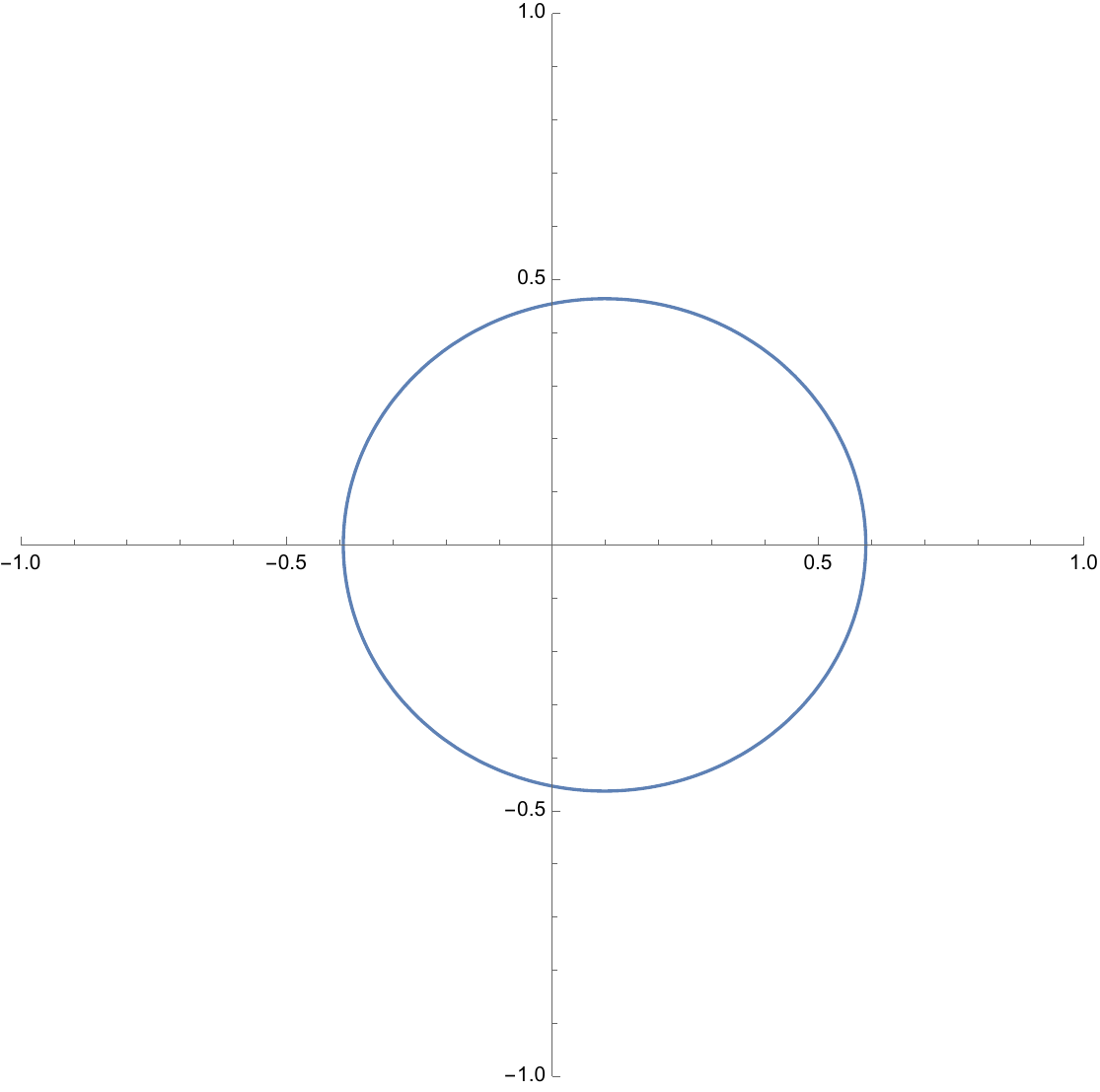} 
    \end{minipage}
    \begin{minipage}{0.45\textwidth}
        \centering
        \includegraphics[width=0.9\textwidth]{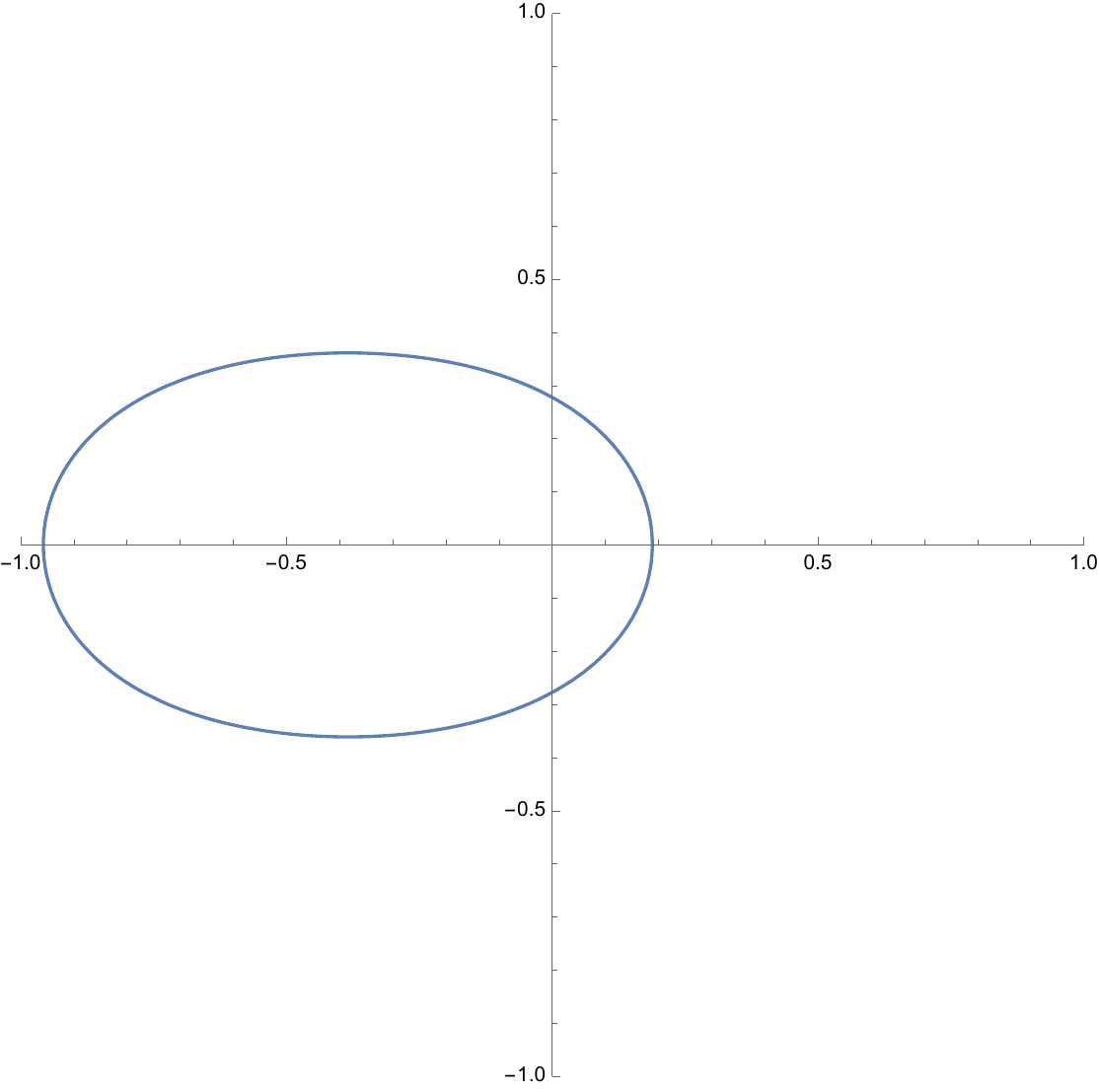} 
    \end{minipage}
    \caption{Two non-disc planar domains with commuting Dirichlet-to-Neumann map and boundary Laplacian.}
\end{figure}

	\noindent Case 3: $p(z)=(z-a)(z-b)$ with $a\neq b$. After integrating, we deduce that up to a constant shift,
	\eq{
		\psi(z)=\frac{1}{a-b}\log\b{\frac{z-a}{z-b}},
	}
	again with $\abs{a}$, $\abs{b}>1$. Note the image of $\psi$ remains unchanged if we precompose it with an isometry of the disc. In particular we may precompose with a rotation to ensure $b\in \R$ without any loss of generality, and then set
	\eq{
		\mu (z)=\frac{z+1/b}{z/b+1},
	}
	an isometry of the disc. One then computes
	\eq{
		\psi(\mu(z))=\frac{1}{a-b}\log\b{\frac{(a-b)z+ab-1}{b^2-1}}
	}
	to the conclude that the image of $\psi$, after rescaling, rotation and translation, is equivalent to Case 2.

	In each case, $\psi$ is injective and hence is a biholomorphism onto $\Omega$. Pulling back the standard metric on $\Omega$ via $\psi$ yields the metric $(\psi')^2g_0$ on the disc, which corresponds exactly to $e^{2\phi}$ on the boundary by construction, thereby concluding the proof.	
\end{proof}

\section{Surfaces of Higher Connectivity}
\label{sec: higher connectivity}

To address the case when the boundary has multiple connected components, we begin with the following proposition.

\begin{prop}
\label{prop: boundary isospectrality}
	If $[\Delta_{\partial M},\Lambda]=0$, then each connected component of $\partial M$ is Laplace-isospectral.
\end{prop}
\begin{proof}
	Decompose $\partial M$ into connected components as $\partial M=\sqcup C_i$ with $i\geq 2$. Say $\phi$ is a $\lambda$-eigenfunction of $\Delta_{C_1}$, define on $\partial M$ the function
	\eq{
		\tilde \phi=\begin{cases}
			\phi&\tn{on }C_1,\\
			0& \tn{on }C_i\tn{ for }i\geq 2,
		\end{cases}
	}
	and let $u\in C^\infty(M)$ be the harmonic extension of $\tilde \phi$. Now, let $\psi=(\Lambda \tilde \phi)|_{C_2}\in C^\infty(C_2)$ and note $\psi$ is nonzero since $u$ cannot vanish both to zeroth and first order simultaneously at the boundary.
	Observe then that
	\eq{
		\Delta_{\partial M}(\Lambda \tilde{\phi})=\Lambda (\Delta_{\partial M}\tilde \phi)=\lambda\cdot \Lambda\tilde \phi\quad\implies\quad \Delta_{C_2}\psi=\lambda \psi,
	}
	so $\lambda$ is an eigenvalue of $\Delta_{C_2}$ with eigenfunction $\psi$. Repeating this process with a basis of the eigenspace on $M_1$ yields that the multiplicity $\lambda$ as an eigenvalue of $C_2$ must be the same as that for $C_1$. This argument applies to every pair of boundary components, and the result follows.
\end{proof}

We note this reasoning applies in any dimensions $d\geq 2$, and as a simple corollary obtain a strengthening of Theorem 1.3 in \cite{GKLP22}:

\begin{proof}[Proof of Theorem \ref{thm: GKLP improvement}]
	It follows from the proof of Theorem 1.3 in \cite{GKLP22} that each connected component of the boundary of $M$ consists of a sphere. By Proposition \ref{prop: boundary isospectrality}, each of these spheres must be isospectral, hence of the same radius. It follows that they must coincide, so the boundary of $M$ is in fact connected and $M$ is therefore a ball.
\end{proof}

This proposition also allows us to address the doubly connected case. Suppose $(M,g)$ is doubly connected, and again apply the smooth Riemann mapping theorem to find $H>0$ and a conformal factor $\phi\in C^\infty(C_H)$ so that $(M,g)$ and $(C_H, g_\phi=e^{2\phi} g_0)$ are isometric, where $g_0$ is the standard flat Euclidian metric on $C_H$ the cylinder of radius $1$ and height $H$.

\begin{prop}
\label{prop: doubly connected commutator}
	If the commutator $[\Delta_{(C_H, g_\phi)},\Lambda_\phi]$ vanishes, then $c_{e^{-2\phi}}(k)=0$ for $\abs{k}\geq 1$, so $\phi$ is locally constant.
\end{prop}
\begin{proof}
	This proof proceeds in similar fashion to that of Proposition \ref{prop: commutator means vanishing coefs}. With natural coordinates $(h,\theta)$ on $C_H$, note that the functions
	\eq{
		f_n(h,\theta)=\frac{e^{n(H-h)}-e^{-n(H-h)}}{e^{nH}-e^{-nH}}e^{in\theta},
	}
	defined for $n\neq 0$, satisfy
	\eq{
		\Delta_0f_n=0,\quad f_n(0,\theta)=e^{in\theta}\quad\tn{and}\quad f_n(H,\theta)=0.
	}
	If the commutator vanishes, we have
	\eq{
		\int_{\partial C_H}e^\phi[e^{-2\phi}(\partial_\theta^2-\phi'\partial_\theta),e^{-\phi}\Lambda_0]f_n\cdot f_{n-k}=0
	}
	for every $n$ and $k$. This integral receives no contribution from the boundary at $\{h=H\}$ since both $f_n$ and $f_{n-k}$ vanish there, and we thus obtain after expanding as before that necessarily
	\eq{
	\label{eq: double conn int}
		0=\b{a(n)(-k^2/2+3kn/2-n^2)-a(k-n)(-n^2+nk/2)}\int_{h=0} e^{ik\theta}e^{-2\phi}\d \theta,
	}
	where
	\eq{
		a(n)=-n\cdot \frac{e^{nH}+e^{-nH}}{e^{nH}-e^{-nH}}.
	}
	The coefficient in Equation (\ref{eq: double conn int}) vanishes for every $n$ only when $k=0$, and the result follows. 
\end{proof}

\begin{proof}[Proof of Theorem \ref{thm: main theorem} in the doubly connected case]
	By Propositions \ref{prop: boundary isospectrality} and \ref{prop: doubly connected commutator}, we see that a doubly connected surface with commuting Dirichlet-to-Neumann map and boundary Laplacian must be $\sigma$-isometric to the cylinder, since the conformal factor must be locally constant at the boundary, and each boundary component must have the same length. Conversely, a quick computation shows that the flat cylinder has this property, concluding the proof.
\end{proof}

It remains to show that the commutativity property can never hold on a surface whose boundary has $k\geq 3$ connected components.

\begin{proof}[Proof of Theorem \ref{thm: main theorem} in the case of $k\geq 3$ boundary components]
	We proceed by contradiction: take $M$ a manifold of genus $g$ as in the statement of the theorem whose boundary has $k\geq 3$ connected components, and suppose $[\Delta_{\partial M}, \Lambda]=0$. Decompose $\partial M$ into connected components as $\partial M=\sqcup C_i$, and note by Proposition \ref{prop: boundary isospectrality} that each is a circle $C_i$ of the same length, which we take to be $2\pi$ without loss of generality. Parametrizing each boundary component by arc length $\theta$, define for $m\in \N$ the function
	\eq{
		f_m=\begin{cases}
			\sin m\theta &\tn{on }C_1,\\
			0 &\tn{otherwise}.
		\end{cases}
	}
	We study the nodal set of the harmonic extension $u_m$ of $f_m$ to $M$. To this end, recall that $M$ naturally inherits a complex structure from the metric, and that a real valued harmonic function may then be realized as the real part of a holomorphic function. Using the standard normal form for holomorphic functions on Riemann surfaces (see Proposition 4.1 in \cite{M95}), we deduce the following properties of the nodal set:
	\begin{itemize}
		\item the nodal set is the union of smooth arcs with endpoints on the boundary,
		\item these arcs may intersect on the interior, but can never merge,
		\item the nodal set branches out of each boundary component exactly $2m$ times, as can be seen by applying the fact that $[\Delta_{\partial M}, \Lambda]=0$ to $f_m$.
	\end{itemize}	
	\begin{figure}[h!]
	\label{fig: graphs}
    	\centering
    	\begin{minipage}{0.85\textwidth}
        	\centering
        	\includegraphics[width=0.9\textwidth]{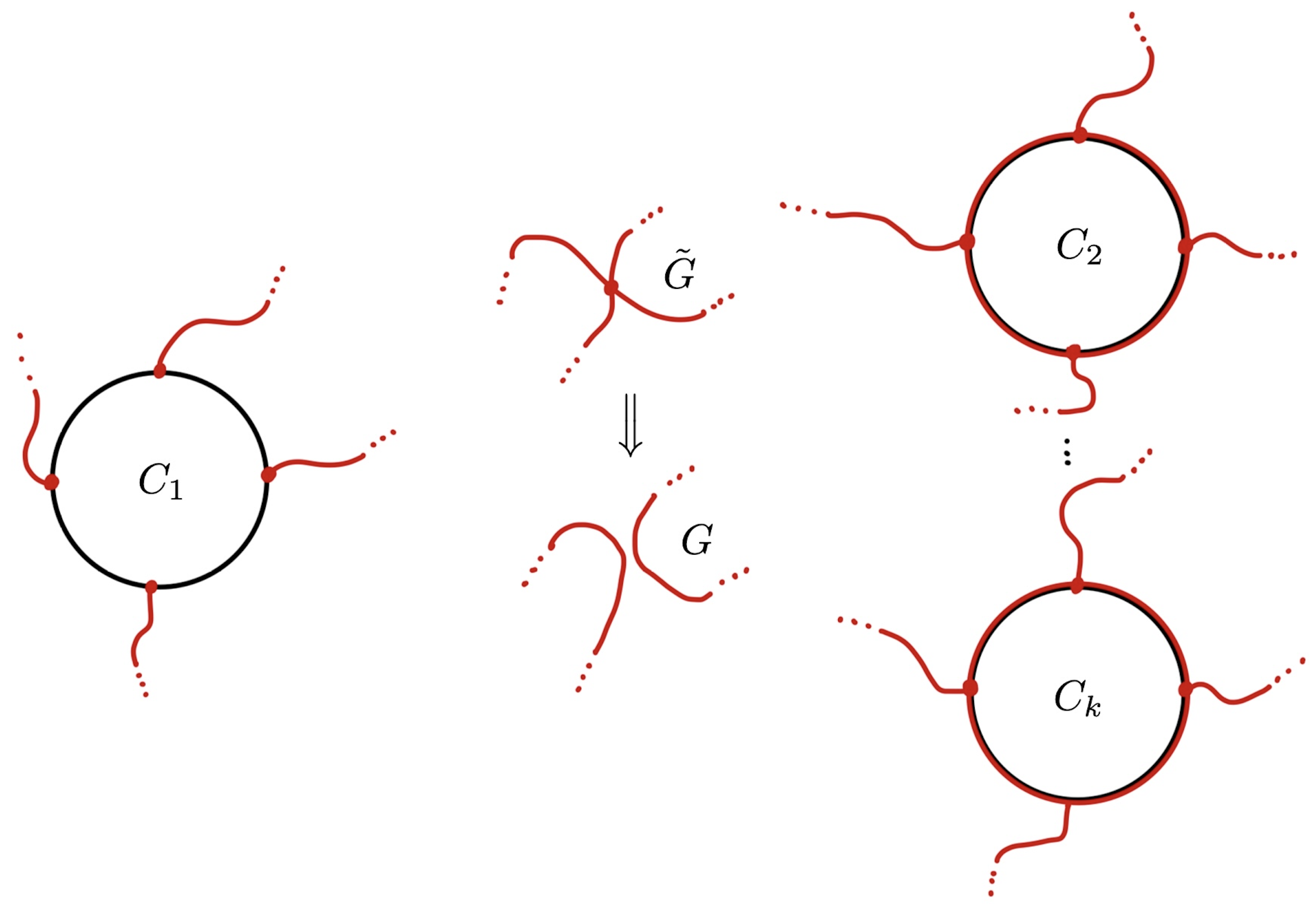}
    	\end{minipage}
   		\caption{Constructing a graph from the nodal set of $u_2$.}
	\end{figure}
	We may therefore associate to the nodal set a graph $\tilde{G}$, whose vertices lie at the intersection points of the nodal set and edges follow arcs of the nodal set. Define then the graph $G$ by removing the vertices of $\tilde{G}$ which lie on the interior of $M$, and successively pairing adjacent edges, as in Figure 2. Now, by construction, $G$ may be embedded on a surface of genus $g$, so with $V$, $E$ and $F$ denoting the number of vertices, edges and faces of this embedding, respectively,
	\eq{
		2-2g\leq V-E+F=2m-mk+F\implies 2-2g+m(k-2)\leq F.
	}
	On the other hand, the number of faces cannot exceed $k$, since otherwise there necessarily exists a region on $M$ bounded entirely by the nodal set, which cannot happen by the maximum principle. We obtain
	\eq{
		2-2g+m(k-2)\leq k,
	}
	and this must hold for \textit{every} positive $m$, a clear contradiciton when $k\geq 3$. Conclude that we cannot possibly have $[\Delta_{\partial M}, \Lambda]=0$ on $M$, as desired.
\end{proof}

\bibliography{bib.bib}{}
\bibliographystyle{plain}

\end{document}